\documentclass[12pt]{article}
\usepackage{amsmath,amscd,amsthm,amsfonts,amssymb}
\begin{document}
\newtheorem*{mainlemma}{Main Lemma}
\newtheorem{thm}{Theorem}
\newtheorem{cor}{Corollary}
\newtheorem{lem}{Lemma}
\newtheorem{slem}{Sublemma}
\newtheorem{prop}{Proposition}
\newtheorem{defn}{Definition}
\newtheorem{conj}{Conjecture}
\newtheorem{ques}{Question}
\newtheorem{claim}{Claim}
\newcounter{constant} 
\newcommand{\newconstant}[1]{\refstepcounter{constant}\label{#1}}
\newcommand{\useconstant}[1]{C_{\ref{#1}}}
\title{On harmonic morphisms from $4$-manifolds to Riemann surfaces and local almost Hermitian structures}
\author{Ali Makki, Marina Ville}
\date{}
\maketitle
\begin{abstract}
We investigate the structure of a harmonic morphism $F$ from a Riemannian $4$-manifold $M^4$ to a $2$-surface $N^2$ near a critical point $m_0$. If $m_0$ is an isolated critical point or if $M^4$ is compact without boundary, we show that $F$ is pseudo-holomorphic  w.r.t. an almost Hermitian structure defined in a neighbourhood of $m_0$. \\
If $M^4$ is compact without boundary, the singular fibres of $F$ are branched minimal surfaces.
\end{abstract}
\section{Introduction}
\subsection{Background}
A harmonic morphism $F:M\longrightarrow N$ between two Riemannian manifolds $(M,g)$ and $(N,g)$ is a map which pulls back local harmonic functions on $N$ to local harmonic functions on $M$. Although harmonic morphisms can be traced back to Jacobi, their study in modern times was initiated by Fuglede and Ishihara who characterized them using the notion of horizontal weak conformality, or semiconformality: 
\begin{defn}\label{defn:horconf} (see [B-W] p.46)
Let $F:(M,g)\longrightarrow (N,h)$ be a smooth map between Riemannian manifolds and let $x\in M$. Then $F$ is called  horizontally weakly conformal at $x$ if either\\
1) $dF_x=0$\\
2) $dF_x$ maps the space $Ker(dF_x)^{\perp}$ conformally onto $T_{F(x)}N$, i.e. there exists a number $\lambda(x)$ called the dilation of $F$ at $x$ such that
$$\forall X,Y\in Ker(dF_x)^{\perp}, h(dF_x(X),dF_x(X))=\lambda^2(x)g(X,Y).$$  
The space $Ker(dF_x)$ (resp. $Ker(dF_x)^{\perp}$) is called the vertical (resp. horizontal) space at $x$.
\end{defn}
Fuglede and Ishihara proved independently
\begin{thm}\label{thm:fugledeishihara}([Fu],[Is])
Let $F:(M,g)\longrightarrow (N,h)$ be a smooth map between Riemannian manifolds. The following two statements are equivalent:\\
1) For every harmonic function $f:V\longrightarrow\mathbb{R}$ defined on an open set $V$ of $N$, the function $f\circ F$ defined on the open set $F^{-1}(V)$ of $M$ is harmonic.\\
2) The map $F$ is harmonic and horizontally weakly  conformal.\\
Such a map is called a harmonic morphism.
\end{thm}
When the target is $2$-dimensional, Baird and Eells proved
\begin{thm}\label{theo:B-E}([B-E])
Let $F:(M^m,g)\longrightarrow (N^2,h)$ be a smooth nonconstant horizontally weakly conformal map between a Riemannian manifold $(M^m,g)$ and a Riemannian $2$-surface $(N^2,h)$. Then $F$ is harmonic (hence a harmonic morphism) if and only if the fibres of $F$ at regular points are minimal submanifolds of $M$.
\end{thm}
It follows from Th.\ref{theo:B-E}  that holomorphic maps from a K\"ahler manifold to a Riemann surface are harmonic morphisms; this raises the question of the interaction between harmonic morphisms to surfaces and holomorphic maps. John Wood studied harmonics morphisms $F:M^4\longrightarrow N^2$ from an Einstein $4$-manifold $M^4$ to a Riemann surface $N^2$ and exhibited an {\it integrable} Hermitian structure $J$ on the regular points of $F$ w.r.t. which $F$ is holomorphic ([Wo]). He extended $J$ to some of the critical points of $F$ and the second author extended it to all critical points ([Vi1]).\\ By contrast, Burel constructed many harmonic morphisms from $\mathbb{S}^4$ to $\mathbb{S}^2$, for non-canonical metrics on $\mathbb{S}^4$ ([Bu]); he was building upon previous constructions on product of spheres by Baird and Ou ([B-O]). Yet it is well-known that $\mathbb{S}^4$ does not admit any global almost complex structure (see for example [St] p.217).
\subsection{The results}
In the present paper, we continue along the lines of [Wo] and [Vi1] and investigate the case of a harmonic morphism $F:M^4\longrightarrow N^2$ from a general Riemannian $4$-manifold $M^4$ to a $2$-surface $N^2$. In [Wo] the integrability of $J$ follows from the Einstein condition so we cannot expect to derive an integrable Hermitian structure in the general case. Could $F$ be pseudo-holomorphic w.r.t. some {\it almost} Hermitian structure $J$ on $M^4$? Burel's example on $\mathbb{S}^4$ tells us that we cannot in general expect $J$ to be defined on all of $M^4$: the most we can expect is for $F$ to be pseudo-holomorphic w.r.t. a local almost Hermitian structure. We feel that this should be true in general; however, we only are able to prove it in two cases:
\newconstant{Z}
\begin{thm}\label{thm:main}
Let $(M^4,g)$ be a Riemannian $4$-manifold, let $(N^2,h)$ be a Riemannian $2$-surface and let $F:M^4\longrightarrow N^2$ be a harmonic morphism. Consider a critical point $m_0$ in $M^4$ and assume that one of the following assertions is true\\
1) $m_0$ is an isolated critical point of $F$\\ OR\\
2) $(M^4,g)$ is compact without boundary (and $m_0$ need not be isolated).\\
Then there exists an almost Hermitian structure $J$ in a neighbourhood of $m_0$ w.r.t. which $F$ is pseudo-holomorphic.
\end{thm}
NB. The  pseudo-holomorphicity of $F$ means: if $m\in M^4$ and $X\in T_m M^4$, 
$$dF(JX)=j\circ dF(X)$$
where $j$ denotes the complex structure on $N^2$.\\
\\
We can use the work of [McD] and [M-W] on pseudo-holomorphic curves: 
under the assumptions of Th. \ref{thm:main}, the local topology of a singularity of a fibre of $F$ is the same as the local topology of a singular complex curve in $\mathbb{C}^2$.\\
\\
We derive from the proof of Th. \ref{thm:main}
\begin{cor}
Let $F:M^4\longrightarrow N^2$ be a harmonic morphism from a compact Riemannian $4$-manifold without boundary to a Riemann surface and let $u_0$ be a singular value of $F$. Then the preimage $F^{-1}(u_0)$ is a (possibly  branched) minimal surface.
\end{cor}
If the manifold $M^4$ is Einstein, the Hermitian structure constructed by Wood is parallel on the fibres of the harmonic morphism and  has a fixed orientation. In the general case, around regular points of $F$, there are two local almost Hermitian structures making $F$ pseudo-holomorphic; they have opposite orientations and we denote them $J_+$ and $J_-$. We follow Wood's computation without assuming $M^4$ to be Einstein and get a a bound on the  product of the $\|\nabla J_{\pm}\|$'s (we had hoped for a local bound on one of the $\|\nabla J_{\pm}\|$'s):
\begin{prop}\label{prop:majo}
Let  $(M^4,g)$ be a Riemannian $4$-manifold, let $(N^2,h)$ be a Riemannian $2$-surface and let $F:M^4\longrightarrow N^2$ be a harmonic morphism. We denote by $j$ the complex structure on $N^2$ compatible with the metric and orientation. For a regular point $m$ of $F$, we let $J_+$ (resp. $J_-$) be the almost complex structure on $T_mM^4$ such that \\
i) $J_+$ and $J_-$ preserve the metric $g$\\
ii) $J_+$ (resp. $J_-$) preserves (resp. reverses) the orientation on $T_mM^4$\\
iii) the map $dF:(T_mM^4, J_{\pm})\longrightarrow (N^2,j)$ is complex-linear.\\
Let $K$ be a compact subset of $M^4$: there exists a constant $A$ such that, for every regular point $m$ of $M^4$ in $K$ and every unit vertical tangent vector $T$ at $m$,
$$\|\nabla_T J_+\| \|\nabla_T J_-\| \leq A$$
when $\nabla$ denotes the connection induced by the Levi-Civita connection on $M^4$.
\end{prop} 

\subsection{Sketch of the paper}
In \S\ref{section:mainlemma}, we recall that the lowest order term of the Taylor development at a critical point of $F$ is a homogeneous holomorphic polynomial; we use it to control one of the two local pseudo-Hermitian structures for which $F$ is pseudo-holomorphic at regular points close to $m_0$. We express this in the Main Lemma (\S \ref{subsection:statementofMainLemma}) and Th. \ref{thm:main} 1) follows almost immediately (\S\ref{section:isolated}). In \S \ref{section:compact} we prove Th. \ref{thm:main} 2) using the twistor constructions of Eells and Salamon ([Ee-Sal]): the twistor space $Z(M^4)$ is a $2$-sphere bundle above $M^4$ endowed with an almost complex structure ${\mathcal J}$ and the regular fibres of $F$ lift to ${\mathcal J}$-holomorphic curves in $Z(M^4)$. The assumptions of Th. \ref{thm:main} 2) enable us to prove that these curves have bounded area so we can use Gromov's compacity theorem: as we approach $m_0$, the lifts of the regular fibres of $F$ in each of the two twistor spaces of $M^4$ converge to a ${\mathcal J}$-holomorphic curve. The Main Lemma enables us to pick one of the two orientations so that the limit curve has no vertical component near $m_0$: near $m_0$, it is the lift of the fibre of $F$ containing $m_0$. This is the key point in the proof of Th. \ref{thm:main} 2).\\
In \S\ref{section:majoration}, we prove Prop \ref{prop:majo} using an identity which Wood established to prove the  superminimality of the fibres  in the Einstein case.
\\ 
\\
For background and detailed information about harmonic morphisms, we refer the reader to  [B-W].
\section*{Acknowledgements} The authors thank Martin Svensson for helpful conversation, Marc Soret for useful comments on previous drafts, Paul Baird and John Wood for answering a crucial question at a crucial moment.\section{The main lemma}\label{section:mainlemma}
\subsection{The almost complex structure at regular points}
A REMARK ABOUT THE NOTATION. If $m$ is a point in $M^4$, we denote by $|m|$ the distance of $m$ to $m_0$. We introduce several constants, which we number $C_1$,...,$C_{10}$,...; they all have the same goal which is to say that one quantity or another is a ${\mathcal O}(|m|)$, so the reader in a hurry can ignore the indices and think of a single constant $C$.\\
\\
Let $m$ be a regular point of $F$ in $M^4$; as we mentioned above in Def. \ref{defn:horconf},  the tangent space of $M^4$ at a regular point $m$ of $F$ splits as follows:
\begin{equation}\label{eq:splitting}
T_mM^4=V_m\oplus H_m
\end{equation}
where the vertical space $V_m$ is the space tangent at $m$ to the fibre $F^{-1}(F(m))$ and the horizontal space $H_m$ is the orthogonal complement of $V_m$ in $T_mM^4$. 
\subsection{The symbol and its extension in a neighbourhood of a critical point}\label{symbolandJ0}
We use the notations of Th.\ref{thm:main} and we let $m_0$ be a critical point of $F$. We denote by $k$, $k>1$, the order of $F$ at $m_0$; namely, if $(x_i)$ is a coordinate system centered at $m_0$, $m_0$ being identified with $(0,...,0)$, we have\\
1) for every multi-index $I=\{i_1,...,i_4\}$ with $|I|\leq k-1$, 
$$\frac{\partial^{|I|}F}{\partial^{i_1}x_1... \partial^{i_4}x_4}(0,...,0)=0$$
2) there exists a multi-index $J=\{j_1,...,j_4\}$ with $|J|=k$ such that 
$$\frac{\partial^{k}F}{\partial^{j_1}x_1... \partial^{j_4}x_4}(0,...,0)\neq 0$$
The lowest order term of the Taylor development of $F$ at $m_0$ is a homogeneous polynomial 
$$P_0:T_{m_0}M^4\longrightarrow T_{F(m_0)}N^2$$ of degree $k$ called the {\it symbol} of $F$ at $m_0$. Fuglede showed ([Fu]) that $P_0$ is a harmonic morphism between $T_{m_0}M^4$ and $T_{F(m_0)}N^2$; it follows from [Wo] that $P$ is a  holomorphic polynomial of degree $k$ for some orthogonal complex structure $J_0$  on $T_{m_0}M^4$. \\
REMARK. The complex structure $J_0$ is not always uniquely defined as the following two examples illustrate:\\
1) $P_0(z_1,z_2)=z_1z_2$: $J_0$ is uniquely defined\\
2) $P_0(z_1,z_2)=z_1^2$: there are two possible $J_0$'s with opposite orientations.
\subsection{The main lemma}\label{subsection:statementofMainLemma}
We identify a neighbourhood $U$ of $m_0$ with a ball in $\mathbb{R}^4$, the point $m_0$ being identified with the origin and we let $(x_i)$ a system of normal coordinates in $U$. We pick these coordinates so that, at the point $m_0$, we have
\begin{equation}\label{eq:structure complexe}
J_0\frac{\partial}{\partial x_1}= \frac{\partial}{\partial x_2}\ \ \ \ 
J_0\frac{\partial}{\partial x_3}= \frac{\partial}{\partial x_4}
\end{equation}
We extend $J_0$ in $U$ by requiring (\ref{eq:structure complexe}) to be verified for all points in $U$.
\newconstant{B}
Of course $J_0$ does not necessarily preserve the metric outside of $m_0$, nevertheless there exists a constant 
$\useconstant{B}$ such that, for a vector $X$ tangent at a point $m$
\begin{equation}\label{eq:complexstructureestimates}
|<J_0X, J_0X>-<X, X>|\leq \useconstant{B}|m|^2\|X\|^2\ \ \ \ |<J_0X, X>|\leq \useconstant{B}|m|^2\|X\|^2
\end{equation}
We identify a neighbourhood of $F(m_0)$ with a disk in $\mathbb{C}$ centered at the origin, with $F(m_0)$ identified with $0$.
We also extend $P_0$ in $U$ by setting
$$P:U\longrightarrow\mathbb{C}$$
$$P(x_1,...,x_4)=P_0(x_1+ix_2, x_3+ix_4).$$
It is clear that for $m\in U$ and $i=1,...,4$
$$\frac{\partial P}{\partial x_i}(m)=\frac{\partial P_0}{\partial x_i}(x(m))$$
hence $P$ is $J_0$-holomorphic.
\newconstant{E}
\begin{mainlemma}
Let $M^4$ be a Riemannian $4$-manifold, $N^2$ a Riemannian $2$-surface and $F:M^4\longrightarrow N^2$ a harmonic morphism. We consider a critical point $m_0$ of $F$ which we do not assume isolated. We denote by $P_0$ the symbol of $F$ at $m_0$, assumed to be holomorphic for a parallel Hermitian complex structure $J_0$ on $T_{m_0}M^4$ and we extend $J_0$ to a neighbourhood of $m_0$ as explained above.\\
In a neighbourhood $U$ of $m_0$, there exists an almost Hermitian structure $J$ continuously defined on the regular points of $F$ in $U$ such that\\
1) $J$ has the same orientation as $J_0$\\
2) $F$ is pseudo-holomorphic w.r.t. $J$.\\
Moreover, for a point $m$ in $U$ 
$$|J(m)-J_0(m)|\leq \useconstant{E}|m|$$
for some positive constant $\useconstant{E}$ independent of $m$.
\end{mainlemma}
\begin{proof}
\newconstant{C}
\newconstant{D}
We let  $\Psi=F-P$. By definition of the symbol, there exist $\useconstant{C},\useconstant{D}$ such that 
\begin{equation}\label{eq:estimee pour psi}
\forall m\in U, \forall X\in T_mM |\Psi(m)|\leq \useconstant{C} |m|^{k+1},\ \ \ |d\Psi(m)(X)|\leq \useconstant{D} |m|^{k}\|X\|
\end{equation}
We let $(\epsilon_1, \epsilon_2)$ be a local positive orthonormal basis of $N^2$ in a neighbourhood of $u_0$. Denoting by $j$ the complex structure on $N$, we have
\begin{equation}\label{eq:jepsilon}
\epsilon_2=j\epsilon_1
\end{equation}
If $m$ is a regular point of $F$, we define two unit orthogonal vectors $e_1, e_2$ in $H_m$ such that
\begin{equation}\label{eq:base horizontale}
dF(e_1)=\lambda(m)\epsilon_1\ \ \ \ dF(e_2)=\lambda(m)\epsilon_2
\end{equation}
where $\lambda(m)$ denotes the dilation of $F$ at $m$ (see Def.\ref{defn:horconf} and [B-W] pp. 46-47).\\
Next we pick  an orthonormal basis $(e_3,e_4)$ of $H_m$ in a way that $(e_1,e_2,e_3,e_4)$ is of the orientation defined by $J_0$. We define the almost complex structure $J$ by setting
\begin{equation}\label{eq:definition de J}
Je_1=e_2\ \ \ \ Je_3=e_4
\end{equation}
We first show that $J_0e_1$ is close to $e_2$; we set
\begin{equation}\label{eq:Jzero3}
J_0e_1=ae_1+be_2+v
\end{equation}
where $a, b\in\mathbb{R}$ and $v\in V_m$.\\  
Since $a=<J_0e_1,e_1>$, we get from (\ref{eq:complexstructureestimates}) 
\begin{equation}\label{eq:estimeedea}
|a|\leq\useconstant{B}|m|^2
\end{equation} 
Next we compute $dF(J_0e_1)$:
$$dF(J_0e_1)=dP(J_0e_1)+d\Psi(J_0e_1)
=jdP(e_1)+d\Psi(J_0e_1)$$
\begin{equation}\label{eq:premieremoitie}
=jdF(e_1)-jd\Psi(e_1)+d\Psi(J_0e_1)
\end{equation}
On the other hand, it follows from (\ref{eq:Jzero3}) that
$$dF(J_0e_1)=adF(e_1)+bdF(e_2)$$
\begin{equation}\label{eq:deuxiememoitie}
=adF(e_1)+jbdF(e_1)
\end{equation}
by definition of $e_1$ and $e_2$ (see (\ref{eq:base horizontale})).\\
Putting (\ref{eq:premieremoitie}) and (\ref{eq:deuxiememoitie}) together, we get
$$j(1-b)dF(e_1)=adF(e_1)+jd\Psi(e_1)-d\Psi(J_0e_1)$$
and using (\ref{eq:base horizontale}), we derive
\begin{equation}\label{eq:jbfe}
|1-b|\lambda(m)=|adF(e_1)+jd\Psi(e_1)-d\Psi(J_0e_1)|
\end{equation}
We already know that the right-hand side of (\ref{eq:jbfe}) is a ${\mathcal O}(|m|^k)$; in order to show that $|1-b|$ is a ${\mathcal O}(|m|)$, we need to bound $\lambda(m)$ below.
\begin{lem}\label{lem:minorerlambda}
There exists a $\useconstant{F}>0$ such that, for $m$ small enough,
$$\lambda(m)\geq \useconstant{F}|m|^{k-1}$$
\end{lem}
\begin{proof}
First we notice  that
\begin{equation}\label{lambda}
\lambda(m)=\sup_{X\in T_mM^4, \|X\|=1} \|dF(m)X\|
\end{equation}
Indeed, take a vector $X\in T_mM$ with $\|X\|=1$. We split it into $X=X_v+X_h$ with $X_v$ vertical and $X_h$ horizontal. Then
$\|X_v\|^2+\|X_h\|^2=1$ and $$\|dF(m)X\|=\|dF(m)X_h\|=\lambda(m)\|X_h\|\leq \lambda(m).$$
\newconstant{F}
\newconstant{G}
 Since $P_0$ is of degree $k$, there exists $\useconstant{G}$ such that for $m$ small enough 
$$\sup_{X\in T_mM^4, \|X\|=1} \|dP(m)X\|\geq \useconstant{G}|m|^{k-1}.$$
It follows that, for $m\in U$ and $X\in T_mM^4$ with $\|X\|=1$, we have
$$|dF(m)X|=|dP(m)X+d\Psi(m)X|\geq |dP(m)X|-|d\Psi(m)X|$$
$$\geq \useconstant{G}|m|^{k-1}-\useconstant{D}|m|^k$$
We take $m$ small enough so that $\useconstant{D}|m|\leq \frac{\useconstant{G}}{2}$ and the lemma follows by  taking
$\useconstant{F}=\frac{\useconstant{G}}{2}$.
\end{proof}
\newconstant{H}
It follows from (\ref{eq:jbfe}) and from Lemma \ref{lem:minorerlambda} that
\newconstant{K}
\begin{equation}\label{eq:estimeedeb2}
|b-1|\leq \useconstant{K}|m|
\end{equation}
for $m$ small enough and some constant $\useconstant{K}$.
\\
To estimate $\|v\|$, we use (\ref{eq:complexstructureestimates}) to write for $m$ small enough
\begin{equation}\label{eq:normedeJe1}
|\|J_0e_1\|^2-1|=|a^2+b^2+\|v\|^2-1|\leq \useconstant{B}|m|^2
\end{equation}
Hence
$$\|v\|^2\leq \useconstant{B}|m|^2+a^2+|b^2-1|$$
\newconstant{L}
and it follows from (\ref{eq:estimeedea}) and (\ref{eq:estimeedeb2}) that 
\begin{equation}\label{estimeedev}
\|v\|\leq\useconstant{L}|m|
\end{equation}
for some positive constant $\useconstant{L}$.\\
We can now conclude. Since 
$$\|Je_1-J_0e_1\|=\|e_2-J_0e_1\|\leq |a|+|b-1|+\|v\|$$
 $\|Je_1-J_0e_1\|$ is a ${\mathcal O}(|m|)$; similarly for $\|Je_2-J_0e_2\|$. \\
\\
We now prove that $\|Je_3-J_0e_3\|$ is a ${\mathcal O}(|m|)$: there are no new ideas so we skip the details. We write
$$J_0e_3=\alpha e_1+\beta e_2+\gamma e_3+\delta e_4$$
Since $(e_i)$ is an orthonormal basis,
$$|\alpha|=| <J_0e_3,e_1>|\leq | <J_0e_1,e_3>|+\useconstant{B}|m|^2$$
using (\ref{eq:complexstructureestimates}); it follows from the estimates above for $J_0e_1
$ that $\alpha$ (and for the same reason $\beta$) is a ${\mathcal O}(|m|)$.\\
We also derive from (\ref{eq:complexstructureestimates}) that
$$|\gamma|=| <J_0e_3,e_3>|\leq  \useconstant{B}|m|^2$$ 
Now that we know that $\alpha, \beta$ and $\gamma$ are ${\mathcal O}(|m|^2)$'s, we focus on $\delta$ and derive from (\ref{eq:complexstructureestimates})
$$|\alpha^2+\beta^2+\gamma^2+\delta^2-1|=|\|J_0e_3\|^2-1|\leq\useconstant{B}|m|^2$$
It follows that $|\delta^2-1|$ is an ${\mathcal O}(|m|^2)$, hence $\delta$ is either close to $1$ or to $-1$:  let us prove that $\delta$ is positive, using orientation arguments.\\
In a neighbourhood of $m$, we identify $\Lambda^4(M)$ with $\mathbb{R}$ so we can talk of signs of $4$-vectors. If we denote by $\star$ the Hodge star operator, the sign of $e_1\wedge J_0e_1\wedge\star (e_1\wedge J_0e_1)$ gives us the orientation of $J_0$ hence, by our assumption, it is of the same sign 
as $e_1\wedge e_2\wedge e_3\wedge e_4$.
 We have seen that, close to $m$, $J_0e_1$ is close to $e_2$, hence $e_1\wedge J_0e_1
\wedge\star (e_1\wedge J_0e_1)$ and $e_1\wedge J_0e_1
\wedge e_3\wedge J_0e_3$ have the same sign; this latter $4$-vector has the same sign as
$\delta e_1\wedge e_2\wedge e_3\wedge e_4$. It follows that $\delta$ is positive.

Hence $\|J_0e_3-Je_3\|$ is a ${\mathcal O}(|m|)$ and so is $\|J_0e_4-Je_4\|$, by identical arguments; this concludes the proof of the Main Lemma.
\end{proof}
\section{Proof of Th.\ref{thm:main} 1): an isolated critical point}\label{section:isolated}
If $m_0$ is an isolated critical point, the almost complex structure $J$ given by the Main Lemma is defined in $U\backslash\{m_0\}$ where $U$ is a neighbourhood of $m_0$. At the point $m_0$, we put $J(m_0)=J_0$ and the Main Lemma tells us that the resulting almost complex structure is continuous.
\section{Proof of Th.\ref{thm:main} 2):$M$ is compact without boundary}\label{section:compact}
\subsection{Background: twistor spaces}
We give here a brief sketch of Eells-Salamon's work with twistors ([Ee-Sal]); the reader can find a more detailed exposition in Chap. 7 of [B-W].\\
The {\it twistor space} $Z^+(M^4)$ (resp. $Z^-(M^4)$) of  an oriented Riemannian $4$-manifold $(M^4,g)$ is the $2$-sphere bundle defined as follows: a point in $Z^+(M^4)$ (resp. $Z^-(M^4)$) is of the form $(J_0,m_0)$ where $m_0$ is a point in $M^4$ and $J_0$ is an orthogonal complex structure on $T_{m_0}M^4$ which preserves (resp. reverses) the orientation on $T_{m_0}M^4$. The twistor spaces $Z^{\pm}(M^4)$ admit the following almost complex structures ${\mathcal J}_{\pm}$. \\
We split the tangent space $T_{(J_0,m_0)}Z^{\pm}(M^4)$ into a horizontal space ${\mathcal H}_{(J_0,m_0)}$ and a vertical space ${\mathcal V}_{(J_0,m_0)}$. Since ${\mathcal H}_{(J_0,m_0)}$ is naturally identified with $T_{m_0}M^4$, we define ${\mathcal J}_{\pm}$ on ${\mathcal H}_{(J_0,m_0)}$ as the pull-back of $J_0$ from $T_{m_0}M^4$; the fibre above $m_0$ is an oriented $2$-sphere so we define ${\mathcal J}_{\pm}$ on ${\mathcal V}_{(J_0,m_0)}$ as the opposite of the canonical complex structure on this $2$-sphere. If $S$ is an oriented $2$-surface in $M^4$, it has a natural lift inside the twistor spaces: a point $p$ in $S$ lifts to the point $(J_p,p)$ in $Z^{+}(M^4)$ (resp. $Z^{-}(M^4)$), where $J_p$ is the orthogonal complex structure on $T_{p}M^4$ which preserves (resp. reverses) the orientation and for which the oriented plane $T_{p}S$ is an oriented complex line.\\
Jim Eells and Simon Salamon proved
\begin{thm}\label{thm:eellssalamon}([Ee-Sal])
Let $(M^4,g)$ be a Riemannian $4$-manifold. A minimal surface in $M^4$ lifts into a ${\mathcal J}_+$-holomorphic (resp. ${\mathcal J}_-$-holomorphic) curve in $Z^+(M^4)$ (resp. $Z^-(M^4)$). Conversely, every non vertical ${\mathcal J}_{\pm}$-holomorphic curve in $Z^{\pm}(M^4)$ is the lift of a minimal surface in $M^4$.  
\end{thm} 
\subsection{Convergence of the twistor lifts of regular fibres}
We assume $M^4$ to be oriented: Th.\ref{thm:main} is local so if $M^4$ is not oriented, we endow a ball centered in at $m_0$ with an. We endow $M^4$ with the orientation given by the complex structure $J_0$ on $T_{m_0}M^4$ defined by the symbol (see \S\ref{symbolandJ0}). From now on we drop the superscript $+$ and we write $Z(M^4)$ for $Z^+(M^4)$. \\
We denote $u_0=F(m_0)$ and we let $(u_n)$ be a sequence of regular values of $F$ which converges to $u_0$. The preimages of the $u_n$'s are smooth compact closed $2$-submanifolds of $M^4$. For every positive integer $n$, we let  $$S_n=F^{-1}(u_n).$$
\begin{lem}\label{lem:boundedareaofSn}
The $S_n$'s all have the same area.
\end{lem}
\begin{proof}
 The singular values of $F$ are discrete so we can assume that the $S_n$'s are all deformation of one another; moreover they are all minimal. It follows from the formula for the first variation of area that if $(\Sigma_t)$, $t\in[0,1]$ is a family of minimal surfaces without boundary in a compact manifold, $\frac{d}{dt}area (\Sigma_t)=0$, hence $area (\Sigma_t)$ is constant, for $t\in [0,1]$.
 \end{proof}
We denote by $\tilde{S}_n$ the lift of $S_n$ into $Z^{+}(M^4)$: Th. \ref{thm:eellssalamon} tells us that they are ${\mathcal J}$-holomorphic curves.  Moreover we have
\begin{lem}\label{lem:boundedarea}
There exists a constant $C$ such that, for every positive integer $n$, $$area(\tilde{S}_n)\leq C.$$
\end{lem}
\begin{proof}
We parametrize the $\tilde{S}_n$'s by maps
$$\gamma_n:S_n\longrightarrow\tilde{S}_n$$
We let $(e_1,e_2)$ be an orthonormal basis of the tangent bundle $TS_n$ and we denote by $\tilde{e}_1,\tilde{e}_2$ their lift in $Z(M^4)$. 
For $i=1,2$, we split  $\tilde{e}_i$ into vertical and horizontal components,
$$\tilde{e}_i=\tilde{e}_i^h+\tilde{e}_i^v$$
We write the area element of $\tilde{S}_n$:
\begin{equation}\label{eq:areaelement}
\|\tilde{e}_1\wedge \tilde{e}_2\|\leq \|\tilde{e}_1^h\wedge \tilde{e}_2^h\|+\|\tilde{e}_1^h\wedge \tilde{e}_2^v\|+\|\tilde{e}_1^v\wedge \tilde{e}_2^h\|+\|\tilde{e}_1^v\wedge \tilde{e}_2^v\|
\end{equation}
Integrating (\ref{eq:areaelement}) and using Cauchy-Schwarz inequality, we get
$$area(\tilde{S}_n)\leq area(S_n)+2\sqrt{area(S_n)}\sqrt{\int_{S_n}\|\nabla\gamma_n\|^2}+\int_{S_n}\|\nabla\gamma_n\|^2$$
where $\nabla$ denote the connection on $Z(M^4)$ induced by the Levi-Civita connection on $M^4$.
\begin{lem}\label{lem:gammabounded}
There exists a constant $A$ such that, for every positive $n$
$$\int_{S_n}\|\nabla \gamma_n\|^2\leq A$$
\end{lem}
\begin{proof} We need to introduce a few notations to give a formula for the integral in Lemma \ref{lem:gammabounded}.
For every $n$, we let $NS_n$ be the normal bundle of $S_n$ in $M^4$ and we endow it with a local orthonormal basis $(e_3,e_4)$. We denote by $R$ the curvature tensor of $(M^4,g)$ and we put
$$\Omega^T=<R(e_1,e_2)e_1,e_2>\ \ \ \ \ \ \ \ \ \Omega^N=<R(e_1,e_2)e_3,e_4>$$
Finally we let  $c_1(NS_n)$ be the degree of $NS_n$ i.e. the integral of its $1$st Chern class; it changes sign with the orientation of $M^4$. Note that other authors (e.g. [C-T]) denote it by  $\chi(NS_n)$, by analogy with the Euler characteristic.\\   
We denote by $dA$ the area element of $S_n$ and we derive from [C-T] (see also [Vi 2])
\begin{equation}\label{nablaJplus} 
\frac{1}{2}\int_{S_n}\|\nabla\gamma_n\|^2=-\chi(S_n)-c_1(NS_n)+\int_{S_n}\Omega^TdA+\int_{S_n}\Omega^NdA
\end{equation}
The critical values of $F$ are isolated, hence the regular fibres all have the same homotopy type and the same homology class $[S_n]$ in 
$H_2(M^4,\mathbb{Z})$. In particular, $|\chi(S_n)|$ does not depend on $n$. The $S_n$'s are embedded hence $c_1(NS_n)$ is equal to the self-intersection number $[S_n].[S_n]$ which does not depend on $n$ either.\\
Since $M^4$ is compact, the expression $|<R(u,v)w,t>|$ has an upper bound for all the $4$-uples of unit vectors 
$(u,v,w,t)$. It follows that the integrals in $\Omega^T$ and $\Omega^N$ in (\ref{nablaJplus}) have a common bound in absolute value.\\
In conclusion, all the terms in the RHS of (\ref{nablaJplus}) are bounded in absolute value uniformely in $n$.
\end{proof}
 Lemma \ref{lem:boundedarea} follows immediately.\\
\end{proof}
Thus the $\tilde{S}_n$'s are ${\mathcal J}$-holomorphic curves of bounded area in $Z(M^4)$: Gromov's result ([Gro]) ensures that they admit a subsequence which converge in the sense of cusp-curves to a ${\mathcal J}$-holomorphic curve $C$. 
\begin{lem}\label{lem:piC} We denote by $\pi:Z(M^4)\longrightarrow M^4$ the natural projection. Then
$$\pi(C)=F^{-1}(u_0).$$
\end{lem}
\begin{proof}
The map $\pi\circ F$ is continuous, so it is clear that $\pi(C)\subset F^{-1}(u_0)$.\\
To prove the reverse inclusion, we take a point $p\in F^{-1}(u_0)$ and we claim
\begin{claim}\label{claim:sequence}
There exists a subsequence $(u_{s(n)})$ of $(u_n)$ and a sequence of points $(p_n)$ of $M^4$ converging to $p$ with $$F(p_n)=u_{s(n)}$$ for every positive integer $n$.
\end{claim}
Indeed, if Claim \ref{claim:sequence} was not true, we would have the following
\begin{claim}\label{claim:false}
$\exists\epsilon>0$ such that $\forall n\in \mathbb{N}^*$ and $\forall m\in M^4$ with $F(m)=u_n$, we have 
$$d(m,p)>\epsilon.$$
\end{claim} 
If Claim \ref{claim:false} was true, the set $F(B(x,\epsilon))$ would contain $u_0$ but would not be a neighbourhood of $u_0$, a contradiction of the fact that a harmonic morphism is open ([Fu],[B-W] p.112).\\
So Claim \ref{claim:sequence} is true: if we denote by $(J_n,p_n)$ the pullback of the $p_n$'s in the twistor lifts $\tilde{S}_n$, they admit a subsequence which converges to a point $(\hat{J},p)$, for some $\hat{J}$ in the twistor fibre above $p$. Clearly $(\hat{J},p)$ belongs to $C$, hence $p$ belongs to $\pi(C)$ and Lemma \ref{lem:piC} is proved.
\end{proof}
Their are a finite number of points $p_1,...,p_k\in M^4$ and positive integers $q_1,...,q_k$ such that the curve $C$ 
can be written
\begin{equation}\label{eq:courbeC}
C=\Gamma+\sum_{i=1}^k q_iZ_{p_i}
\end{equation}
where $\Gamma$ is a ${\mathcal J}$-holomorphic curve with no vertical components and the $Z_{p_i}$'s are the twistor fibres above the $p_i$'s.
It follows from Lemma \ref{lem:piC} that 
$$\pi(\Gamma)=F^{-1}(u_0).$$
We derive that $F^{-1}(u_0)$ is a minimal surface possibly with branched points and having $\Gamma$ as its twistor lift.\\
Note that the presence of twistor fibres in (\ref{eq:courbeC}) is to be expected: when a sequence of smooth minimal surfaces converges to a minimal surface with singularities, its twistor lifts can experience bubbling off of twistor fibres above singular points (see [Vi2] for a more detailed discussion of this phenomenon). However, in the present case, the Main Lemma excludes such bubbling-off in a neighbourhood of $m_0$:
\begin{lem}\label{lem:nobubbles}
There exists an $\epsilon>0$ such that, if $q_i$ is one of the points appearing in (\ref{eq:courbeC}), 
$$dist(m_0,p_i)>\epsilon$$
\end{lem}
\begin{proof}
Since the $q_i$'s are finite in number, it is enough to prove that $m_0$ is not one of them.\\
The almost complex structure $J_0$ appearing in the Main Lemma does not necessarily preserve the metric outside of $m_0$; so we introduce  the bundle ${\mathcal C}$ of {\it all} the complex structures on $TM^4$ which preserve the orientation. It contains the bundle $Z(M^4)$ and embeds into the bundle $GL(TM^4)$. We denote by $d_{\mathcal C}$ the distance on ${\mathcal C}$ induced by the metric on $GL(TM^4)$ and by $d_{M^4}$ the distance in $M^4$.
\begin{lem}\label{lem:Jmpetit}
$\forall\epsilon>0\ \ \ \exists\eta>0\ \ \ $ such that 
$$d_{M^4}(m,m_0)<\eta\Rightarrow d_{\mathcal C}[(J(m),m),(J_0,m_0)] <\epsilon$$
\end{lem}
\begin{proof} $d_{\mathcal C}[(J(m),m),(J_0,m_0)]$
\begin{equation}\label{eq:majo}
\leq d_{\mathcal C}[(J(m),m),(J_0(m),m)]+d_{\mathcal C}[(J_0(m),m),(J_0,m_0)]
\end{equation}
We bound the first term in (\ref{eq:majo}) using the Main Lemma; the second term is bounded because  $J_0:U\longrightarrow {\mathcal C}$ is continuous.
\end{proof}
If $m$ is a regular point of $F$, we denote by $\gamma(m)$ the point above $m$ in the twistor lift of $F^{-1}(F(m))$; in the Main Lemma, we defined the almost complex structure $J(m)$. The tangent plane to the fibre at $m$ is a complex line for both $\gamma(m)$ and $J(m)$; since $\gamma(m)$ and $J(m)$ both preserve the orientation, it follows that $\gamma(m)=\pm J(m)$. We can get rid of the $\pm$ by saying that $F^{-1}(u_0)$ is a $2$-dimensional CW-complex, hence $B(m_0,\epsilon)\backslash F^{-1}(u_0)$ is connected: there is a $s\in\{-1,+1\}$ such that for every regular point $m$ of $F$ near $m_0$, 
\begin{equation}\label{eq:sHJ} 
\gamma(m)=sJ(m)
\end{equation}
We rewrite Lemma \ref{lem:Jmpetit}:
$\forall\epsilon>0\ \ \ \exists\eta>0\ \ \ $ such that for a regular point $m$, 
\begin{equation}\label{eq:estimeeH}
d_M(m,m_0)<\eta\Rightarrow d_{\mathcal C}[(\gamma(m),m),(sJ_0,m_0)] <\epsilon
\end{equation}
If the whole twistor fibre $Z_{m_0}M^4$ was included in $C$, it would be in the closure of the union of the twistor lifts of the regular fibres of $F$ in a neighbourhood of $m_0$: we see from (\ref{eq:estimeeH}) that this is impossible. This concludes the proof of Lemma \ref{lem:nobubbles}.
\end{proof}
\subsection{Construction of the almost complex structure}
We now construct a local section of $Z(M^4)$, for which $F$ holomorphic. As in [Vi1], we work first on the space $\mathbb{P}(Z(M^4))$ obtained by taking the quotient of each twistor fibre by its antipody; if $J$ is an element of $Z(M^4)$, we denote by $\bar{J}$ its image in $\mathbb{P}(Z(M^4))$.\\
If $m$ is a regular point of $F$, there are $2$ complex structures, $J_1$ and $J_2$, on $T_{m_0}M^4$ for which the unoriented planes $V_{m}$ and $H_{m}$ are complex lines. These two complex structures verify $J_1=-J_2$, hence they define the same point, denoted $\bar{J}(m)$, in $\mathbb{P}(Z_m(M^4))$. To extend this section of $\mathbb{P}(Z(M^4))$ above $F^{-1}(u_0)$, we state
\begin{lem}\label{lem:antipody}
There exists an $\epsilon>0$ such that every $m\in B(m_0,\epsilon)\cap F^{-1}(u_0)$ has either a single preimage in $\Gamma$ or exactly two antipodal preimages in $\Gamma$.
\end{lem}
\begin{proof} 
We let $\epsilon$ be a number satisfying Lemma \ref{lem:nobubbles} and we pick $m\in B(m_0,\epsilon)\cap F^{-1}(u_0)$. Since $\Gamma$ has no vertical component above $B(m,\epsilon)\cap F^{-1}(u_0)$, it meets $Z_{m}(M^4)$  at a discrete number of points. Let us assume that $J_1$ and $J_2$ are two different elements of $\Gamma\cap Z_{m}(M^4)$. There exist two non vertical possibly branched disks $\Delta_1$ and $\Delta_2$ in $\Gamma$ containing $(J_1,m)$ and $(J_2,m)$ respectively. Each one of the two $\Delta_i$'s is the twistor lift of a possibly branched disk $D_i$ of $F^{-1}(u_0)$. The disks $D_1$ and $D_2$ meet at $m$: if they have different tangent planes at $m$, this implies that $m$ is a singular point of $F^{-1}(u_0)$. Since  $F^{-1}(u_0)$ is a closed minimal surface, its singular points are discrete so we can make $\epsilon$ small enough so that there is not singular point in $F^{-1}(u_0)\cap B(m_0,\epsilon)$ except for possibly $m_0$.\\
So we assume that $m_0$ is a singular point of $F^{-1}(u_0)$. Because the symbol is $J_0$ holomorphic, all planes tangent to $m_0$ at $F^{-1}(u_0)$ are $J_0$-complex lines and it follows that $J_1=-J_2=\pm J_0$.
\end{proof} 
We denote by $\bar{\Gamma}$ the projection of $\Gamma$ in $\mathbb{P}(Z(M^4))$ and  
 by $\bar{J}$ the local section of $\bar{\Gamma}$ given by Lemma \ref{lem:antipody}.
\begin{lem}\label{lem:continuousinPZ}
There exists a small $\epsilon>0$ such that the map
$$B(m_0,\epsilon)\longrightarrow \mathbb{P}(Z(M^4))$$
$$m\mapsto \bar{J}(m)$$
is continuous.
\end{lem}
\begin{proof}
Since $\bar{J}$ is continuous above $U\backslash F^{-1}(u_0)$, we consider a sequence of points $(p_n)$ in $M^4$ converging to a $p_0$ with $F(p_0)=u_0$. It is enough to the consider two cases\\
 1st case: all the $F(p_n)$'s are regular values\\
 2nd case: for every $n$, $F(p_n)=u_0$.\\
If $(p_n)$ is a general sequence, we extract subsequences of the form 1) or 2).\\
{\bf 1st case} - For every $n$, $v_n=F(p_n)$ is a regular value of $F$.\\
i) First assume that $u_n=v_n$ for every $n$. Since $\Gamma$ is the limit of the twistor lifts of the $F^{-1}(u_n)$ the sequence 
$(\bar{J}(p_n),p_n)$ converges to a point $(\bar{K},p_0)$ in $\bar{\Gamma}$; Lemma \ref{lem:antipody} ensures that $\bar{K}=\bar{J}(p_0)$.\\
ii) In the general case, the $v_n$'s converge to $u_0$ so we can proceed with the $v_n$'s as we did with the $u_n$'s and derive that the twistor lifts of the $F^{-1}(v_n)$'s converge in the sense of Gromov to the twistor lift of $F^{-1}(u_0)$ and conclude as in i).\\
{\bf 2nd case}  For every $n$, $F(p_n)=u_0$. We denote  by $\bar{\pi}$ the natural projection from $\mathbb{P}(Z(M^4))$ to $M^4$. Lemma \ref{lem:antipody} ensures that $\bar{\pi}$ restricts to a continuous bijection from $\bar{\Gamma}\cap \bar{\pi}^{-1}(\bar{B}(m_0,\frac{\epsilon}{2}))$ to $F^{-1}(u_0)\cap \bar{B}(m_0,\frac{\epsilon}{2})$; since these spaces are compact and Hausdorff, a continuous bijection between them is a homeomorphism (see for example [Han] p. 45). It follows that, if the $p_n$'s converge to $p_0$, their preimages in $\bar{\Gamma}$ converge to the preimage of $p_0$; in other words, the $\bar{J}(p_n)$'s converge to $\bar{J}(p_0)$.
\end{proof}
We conclude as in [Vi1]. We lift $\bar{J}$ above the set of regular points by taking for $J$ the one complex structure on $T_mM$ which renders $dF$ holomorphic at that point - this requirement defines it uniquely on the horizontal space $H_m$ and since, the orientation of $J$ is given, there is also a unique possibility for $J$ on $V_m$. By the same argument as in [Vi1], this extends to the entire $B(m_0,\epsilon)$.\\
This concludes the proof of Th.\ref{thm:main} 

\section{Proof of Prop. \ref{prop:majo}}\label{section:majoration}
We begin by reproducing part of Wood's arguments ([Wo]).\\
We let $m$ be a regular point of $F$ and we denote by $V_m$ (resp. $H_m$) the vertical (resp. horizontal) space at $m$. We let $S_0V_m$ be the set of symmetric trace-free holomorphisms of $V_m$ and we define the Weingarten map
$$A:H_m\longrightarrow S_0V_m$$
$$X\mapsto (U\mapsto \nabla^V_UX)$$
where $\nabla^V_UX$ denotes the vertical projection of $\nabla_UX$. \\
\newconstant{L}
At a regular point $m$, we denote by $J_+$ (resp. $J_-$) the Hermitian structure on $T_mM^4$ w.r.t. which $dF:T_mM^{4}\longrightarrow T_{F(m)}N^2$ is $\mathbb{C}$-linear and which preserves (resp. reverses) the orientation on $T_mM^4$.
If $M^4$ is Einstein, Wood proves in Prop. 3.2 that all horizontal vectors $X$ verify $$^tA\circ A(J_{\pm}X)=J(^tA\circ A)(X).$$ If $M^4$ is not Einstein, we follow his proof to derive the existence of $\useconstant{L}$ such that, for every unit horizontal vector $X$ tangent to a regular point of $F$ in $K$,
\begin{equation}\label{equation:AcircA}
\|^tA\circ A(J_{\pm}X)-J_{\pm}(^tA\circ A)(X)\|\leq\useconstant{L}
\end{equation}
We now put $T=e_1$ and we complete it into an orthonormal basis $(e_1,e_2)$ of $V_m$; we pick an orthonormal basis  $(e_3,e_4)$ of $H_m$ such that the almost complex structures verify
\begin{equation}\label{equation:expressiondeJpm}
e_2=J_+e_1=-J_-e_1\ \ \ \ \ \ \ \ \ e_4=J_+e_3=J_-e_3
\end{equation}
We let $E_1$ and $E_2$ be the following elements of $S_0V_m$ defined by their matrices in the base $(e_1,e_2)$.
$$E_1 =
 \begin{pmatrix}
  1 & 0  \\
  0 & -1 
  \end{pmatrix}\ \ \ \ E_2 =
 \begin{pmatrix}
  0 & 1  \\
  1 & 0 
  \end{pmatrix}$$
We write the matrix of $A$ in the bases $(e_3,e_4)$ and $(E_1,E_2)$
$$A =
 \begin{pmatrix}
  a & b  \\
  c & d 
  \end{pmatrix}$$
 where 
 \begin{equation}\label{equation:matricedeAab}
 a=-<\nabla_{e_1}e_1,e_3>\ \ \ \ \ \ \ \ \  b=-<\nabla_{e_1}e_1,e_4>
 \end{equation}
\begin{equation}\label{equation:matricedeAcd}
 c=-<\nabla_{e_1}e_2,e_3>\ \ \ \ \ \ \ \ \  d=-<\nabla_{e_1}e_2,e_4>
 \end{equation}
The homomorphisms $J_+$ and $J_-$ coincide on the basis $(e_3,e_4)$ (see (\ref{equation:expressiondeJpm})); we compute  
$$(^tA\circ A)J_{\pm}
-J_{\pm}(^tA\circ A) =
 \begin{pmatrix}
  2(ab+cd) & b^2+d^2-(a^2+c^2)  \\
  b^2+d^2-(a^2+c^2) & -2(ab+cd)
  \end{pmatrix}$$
  and we derive from (\ref{equation:AcircA}) 
  \begin{equation}\label{equation:AcircAbis}
  |ab+cd|\leq \useconstant{L}\ \ \ \ \ \ \ \ \ \  |b^2+d^2-(a^2+c^2)|\leq\useconstant{L}
  \end{equation}
  We take $J$ to be $J_+$ or $J_-$ and we write the Euclidean norm
  \begin{equation}\label{equation:normedenablaJ}
  \|\nabla_{e_1}J\|^2=\sum_{i,j=1,...,4}<(\nabla_{e_1}J)e_i,e_j>^2
  \end{equation}
  \begin{equation}
  =\sum_{i,j=1,...,4}(<\nabla_{e_1}(Je_i),e_j>-<J\nabla_{e_1}e_i,e_j>)^2
  \end{equation}
  \begin{equation}\label{equation:formulenablaJ1}
 =\sum_{i,j=1,...,4}(<\nabla_{e_1}(Je_i),e_j>+<\nabla_{e_1}e_i,Je_j>)^2
  \end{equation}
  It is  enough to take $e_i$ vertical and $e_j$ horizontal in (\ref{equation:formulenablaJ1}):
 \begin{lem}\label{lem:normedenablaJ1}
  $\|\nabla_{e_1}J\|^2=2\sum_{1\leq i\leq 2,3\leq j\leq 4} (<\nabla_{e_1}(Je_i),e_j>+<\nabla_{e_1}e_i,Je_j>)^2$
\end{lem}
\begin{proof}
  If $e_i$ and $e_j$ are both horizontal or both vertical, Prop. 2.5.16 i) of [B-W] yields
  \begin{equation}\label{equation:bothvertorhor}
  <\nabla_{e_1}(Je_i),e_j>=<J\nabla_{e_1}e_i,e_j> 
  \end{equation}
  Note that Baird-Wood's Prop. 2.5.16 is about horizontal vectors, but its proof works identically for vertical vectors.\\
  Assume now that $e_i$ is horizontal and $e_j$ is vertical: 
  \begin{equation}\label{equation:vertiandhori}
  <\nabla_{e_1}(Je_i),e_j>+ <\nabla_{e_1}e_i,Je_j>=-<Je_i,\nabla_{e_1}e_j>-<e_i,\nabla_{e_1}(Je_j)>
  \end{equation}
  Putting together (\ref{equation:formulenablaJ1}), (\ref{equation:bothvertorhor}) and (\ref{equation:vertiandhori}) completes the proof of Lemma \ref{lem:normedenablaJ1}.
\end{proof}
We use the values given for the $J_{\pm}$ in (\ref{equation:expressiondeJpm}) to derive $$\frac{1}{2}\|\nabla_{e_1}J_{\pm}\|^2=(\pm<\nabla_{e_1}e_2,e_3>+<\nabla_{e_1}e_1,e_4>)^2$$
$$+
(\pm<\nabla_{e_1}e_2,e_4>-<\nabla_{e_1}e_1,e_3>)^2$$
$$+(\pm<\nabla_{e_1}e_1,e_3>-<\nabla_{e_1}e_2,e_4>)^2$$
\begin{equation}\label{equation:grosseformule}
+(\pm<\nabla_{e_1}e_1,e_4>+<\nabla_{e_1}e_2,e_3>)^2
\end{equation}
We rewrite (\ref{equation:grosseformule}) in terms of the coefficients $a,b,c,d$ of the matrix $A$ introduced above (see (\ref{equation:matricedeAab}) and (\ref{equation:matricedeAcd})); we get after a short computation 
\begin{equation}
\|\nabla_{e_1}J_+\|^2=4[(a-d)^2+(b+c)^2]=4[a^2+b^2+c^2+d^2-2(ad-bc)]
\end{equation}
\begin{equation}
 \|\nabla_{e_1}J_-\|^2=4[(a+d)^2+(b-c)^2]=4[a^2+b^2+c^2+d^2+2(ad-bc)]
\end{equation}
hence 
\begin{equation}\label{equation:produitdesnablaJ}
\|\nabla_{e_1}J_+\|^2\|\nabla_{e_1}J_-\|^2=16[(a^2+b^2+c^2+d^2)^2-4(ad-bc)^2]
\end{equation}
We now bound  (\ref{equation:produitdesnablaJ}) using (\ref{equation:AcircAbis}). To this effect we put
\begin{equation}\label{equation:R1R2}
a=R_1\cos\theta\ \ \ \ \    c=R_1\sin\theta\ \ \ \ \   b=R_2\cos\alpha\ \ \ \ \   d=R_2\sin\alpha
\end{equation}
and we rewrite (\ref{equation:produitdesnablaJ}) as 
\begin{equation}
\frac{1}{16}\|\nabla_{e_1}J_+\|^2\|\nabla_{e_1}J_-\|^2=(R_1^2+R_2^2)^2-4R_1^2R_2^2\sin^2(\theta-\alpha)
\end{equation}
\begin{equation}
=(R_1^2+R_2^2)^2-4R_1^2R_2^2+4R_1^2R_2^2\cos^2(\theta-\alpha)
\end{equation}
\begin{equation}\label{equation:conclusion}
=(R_1^2-R_2^2)^2+4R_1^2R_2^2\cos^2(\theta-\alpha)
\end{equation}
We now rewrite (\ref{equation:AcircAbis}) as
\begin{equation}\label{equation:majodesRi}
|R_1R_2\cos(\theta-\alpha)|\leq \useconstant{L} \ \ \ \ \ \ \ \ \ |R_1^2-R_2^2|\leq \useconstant{L}
\end{equation}
and this allows us to bound (\ref{equation:conclusion}) and conclude the proof of of Prop. \ref{prop:majo}.\\
$\square$

\bigskip
\footnotesize{Ali.Makki@lmpt.univ-tours.fr,\ \ Marina.Ville@lmpt.univ-tours.fr\\LMPT,
 Universit\'e de Tours
 UFR Sciences et Techniques
 Parc de Grandmont
 37200 Tours, FRANCE}

\end{document}